\documentclass[reqno, 11pt]{amsart}

\usepackage{hyperref}
\usepackage{graphicx}
\usepackage[latin1]{inputenc}
\usepackage[english,activeacute]{babel}
\usepackage{xcolor}

\numberwithin{equation}{section}
\newtheorem{theorem}{Theorem}[section]
\newtheorem{lemma}[theorem]{Lemma}

\newtheorem{proposition}[theorem]{Proposition}
\newtheorem{corollary}[theorem]{Corollary}

\title[Standing waves for NLS on graphs]
{ Orbital stability of standing waves for supercritical NLS with potential on graphs}

\email{ardila@impa.br}

\subjclass[2010]{Primary 76B25, 35Q51, 35Q55, 35R02}
\keywords{Non-linear Schr\"odinger equation; quantum graphs; standing waves; stability}

\begin{document}
\maketitle

\centerline{\scshape Alex H. Ardila}
\medskip
{\footnotesize
 \centerline{ICEx, Universidade Federal de Minas Gerais, Av. Antonio Carlos, 6627}
\centerline{Caixa Postal 702, 30123-970, Belo Horizonte-MG, Brazil}
}

\begin{abstract}
In this paper we study the existence and stability of normalized standing waves for the nonlinear Schr\"odinger equation on a general starlike graph with potentials. Under general assumptions on the graph and the potential, we show the existence of orbitally stable standing waves  when the nonlinearity is $L^{2}$-critical and supercritical. 

\end{abstract}

\section{Introduction}
In this paper we are interested in the existence and stability of standing waves with prescribed $L^{2}$-norm  
for the following nonlinear Schr\"odinger equation on a  metric graph $\Gamma$:
 \begin{equation}
\label{00NL}
 i\partial_{t}u=Hu-|u|^{p-1}u, \quad x\in \Gamma.
\end{equation}
We recall that the nonlinearity in \eqref{00NL} is understood componentwise. {In this paper we are interested in the $L^{2}$-critical and supercritical cases; so we restrict our discussion to the cases where $p\geq 5$.}

Equation \eqref{00NL} models propagation through junctions in networks (see \cite{QGAH, PDEOM, User}). 
The study of nonlinear propagation on graphs/networks is a topic of active research in several
branches of pure and applied science. Modern applications of partial differential equations on networks include nonlinear electromagnetic
pulse propagation in optical fibers, the hydrodynamic, biology, etc; see, e.g., \cite{PDEOM} and the references therein.

Recently numerous results on existence and stability of standing
waves,  local well-posedness of initial value problem, blow up and scattering results for nonlinear Schr\"odinger equation on a metric
graph were obtained. Among such works, let us mention \cite{AA2017, CaFiNo2017, Cacci2017, FSOH, OSCE, FNCN, ADAM, Adaminew, NQTG, AQFF,ADN1, ADNP,AESM, DSED, NNOOA}. In particular, the NLS on the real line with a point interaction  (which can be understood as a metric graph with only two edges) has been also studied substantially in the literature (see \cite{AHdelta,RJJ,FO,LFF}). We refer to \cite{ASPT2} for further information and bibliography.

Let $\Gamma$ be a connected finite metric graph, by $V$ we denote the set of its vertices, and by
$J$ we denote the set of its edges. We will assume $\Gamma$ has at least one external edge, so that $\Gamma$ is noncompact. If an edge $e\in J$ emanates from a vertex $v\in V$, then we will write this as follows: $e\prec v$. 
The differential operator under consideration in this paper  is the Schr\"odinger operator $H$
in $L^{2}(\Gamma)$ equipped with delta conditions concentrated at all vertices $v\in V$ of the graph:
\begin{align}\label{NN1}
(Hu)_{e}&=-u^{\prime\prime}_{e}+W_{e}u_{e}            \\ \label{NN2}
\text{dom}(H)&=\Big\{   u\in H^{2}(\Gamma): \sum_{e\prec v}\partial_{o}u_{e}(v)=-\alpha_{v}u_{e}(v)\quad \text{for all $v\in V$}\Big\}, 
\end{align} 
where { $\alpha_{v}$ are real constants associated with the delta potentials concentrated at vertices $v\in V$. 
Here, as elsewhere, the Sobolev space $H^{1}(\Gamma)$ is defined as the space of continuous functions on $\Gamma$ that belong to $H^{1}(I_{e})$ on each edge, i.e.,
\begin{equation*}
H^{1}(\Gamma)= \left\{u\in C(\Gamma): u_{e}\in H^{1}(I_{e})\,\, \text{for all $e\in J$}\right\}, 
\end{equation*}
where $C(\Gamma)$ is the set of continuous functions on $\Gamma$, and the corresponding norm defined by
\begin{equation*}
\left\|u\right\|^{2}_{H^{1}}=\sum_{e\in J}\|u_{e}\|^{2}_{H^{1}},
\end{equation*}
and  $H^{2}(\Gamma)$ denotes the Sobolev space 
\begin{equation*}
H^{2}(\Gamma)=\left\{u\in H^{1}(\Gamma): u_{e}\in H^{2}(I_{e}), \,\, \text{for all $e\in J$}\right\}.
\end{equation*}
Finally, we denote by $\partial_{o}u_{e}(v)$ the outward derivative of $u$ at $v$ along the edge $e$.}

Formally, the NLS \eqref{00NL} has the following conserved quantity, 
\begin{equation*}
E(u)=\frac{1}{2}\int_{\Gamma}|u^{\prime}|^{2}dx+\frac{1}{2}\int_{\Gamma}W(x)|u|^{2}dx-\frac{1}{2}\sum_{v\in V}\alpha_{v}|u(v)|^{2}-\frac{1}{p+1}\int_{\Gamma}|u|^{p+1}dx.
\end{equation*}
The potential $W(x)$ can be thought of as modeling inhomogeneities in the medium. 

Following \cite{Cacci2017, CaFiNo2017}, for our analysis we make  the following assumptions about the metric  graph $\Gamma$ and the Schr\"odinger operator $H$.\\
\noindent\textbf{Assumption 1.} $\Gamma$ is a finite, connected metric graph, with at least one external edge.\\
We recall that a metric graph is a graph $\Gamma$ equipped with a function $L: J\rightarrow (0,+\infty]$  such that each edge $e\in J$  is identified with a finite segment $[0, L_{e}]$ of positive length $L_{e}$ or an infinite segment $[0, L_{e})$ with $L_{e}=+\infty$. 
Naturally we have the decomposition $J=J^{int}\cup J^{ex}$, where the set $J^{ex}$ denotes the
set of external edges of $\Gamma$ and the set $J^{int}$ denotes the set of internal edges.

\noindent\textbf{Assumption 2.} { $W=W_{+}-W_{-}$ with $W_{\pm}\geq 0$,  $W_{+}\in L^{1}(\Gamma)+L^{\infty}(\Gamma)$, and $W_{-}\in L^{r}(\Gamma)$ for some $r\in [1, 1+2/(p-1)]$}.\\
We remark that assumption 2  implies that the operator $H$ on the graph $\Gamma$ admits a precise interpretation as self-adjoint operator
on $L^{2}(\Gamma)$ (see \cite[Remark 2.1]{CaFiNo2017} for more details). Denote by $\mathfrak{F}[u]$  the quadratic form associated with the operator $H$,
\begin{equation*}
\mathfrak{F_{}}[u]:=\int_{\Gamma}|u^{\prime}|^{2}dx+\int_{\Gamma}W(x)|u|^{2}dx-\sum_{v\in V}\alpha_{v}|u(v)|^{2},
\end{equation*}
defined on the domain $\mathrm{dom}(\mathfrak{F})=H^{1}(\Gamma)$. Let
\begin{equation*}
-\lambda_{0}:=\text{inf}\left\{\mathfrak{F}[u]: u\in H^{1}(\Gamma), \|u\|^{2}_{L^{2}}=1\right\}.
\end{equation*}
\noindent\textbf{Assumption 3.} $\lambda_{0}>0$ and it is an isolated eigenvalue.

Notice that these assumptions are satisfied in many interesting cases;  see introduction in \cite{CaFiNo2017} for more details. Our work is motivated by the recent papers \cite{Cacci2017, CaFiNo2017},  where the orbital stability of standing waves of \eqref{00NL} on a general starlike graph with potentials is considered, with a special focus on the $L^{2}$-subcritical and critical case.

Local well-posedness of the Cauchy problem for \eqref{00NL} in the energy space $H^{1}(\Gamma)$ is established in Cacciapuoti et al.\cite[Propositions 2.3 and 2.8]{CaFiNo2017} for any $p>1$.
\begin{proposition} \label{PCS}
If Assumptions 1-2 hold true, for any $u_{0}\in H^{1}(\Gamma)$, there exist $T=T(u_{0})>0$ and a unique maximal solution 
$u\in C([0,T),H^{1}(\Gamma))$ of \eqref{00NL} with  $u(0)=u_{0}$ such that the following ``blow up alternative'' holds: either $T=\infty$
	or $T<\infty$ and $\mathrm{lim}_{t\rightarrow T}\|u(t)\|_{H^{1}}=\infty$. Furthermore, the conservation of energy and charge hold; that is, 
\begin{equation*}
E(u(t))=E(u_{0})\quad  and \quad \left\|u(t)\right\|^{2}_{L^{2}}=\left\|u_{0}\right\|^{2}_{L^{2}}\quad  \text{for all $t\in [0,T)$}.
\end{equation*}
\end{proposition}
If $1<p<5$, the global well-posedness of the Cauchy problem for \eqref{00NL} holds in $H^{1}(\Gamma)$ by Gagliardo-Nirenberg estimates, conservation of the $L^{2}$-norm and energy; see \cite[Theorem 3]{CaFiNo2017} for more details.

{
We recall the notion of stability of a set $\mathcal{M}\subset H^{1}(\Gamma)$ (see \cite[Chapter 8]{CB} for review of this theory). 
For $\mathcal{M}\subset H^{1}(\Gamma)$, we say that the set $\mathcal{M}$ is $H^{1}(\Gamma)$-stable with respect to NLS \eqref{00NL} if for arbitrary $\epsilon>0$ there exists $\delta>0$ such that if $u_{0}\in H^{1}(\Gamma)$  satisfies $$\inf_{\varphi\in \mathcal{M}}\|u_{0}-\varphi\|<\delta,$$ then
\begin{equation*}
\sup_{t\in  \mathbb{R}}\inf_{\varphi\in \mathcal{M}}\|u(t)-\varphi\|_{H^{1}}<\epsilon,
\end{equation*}
where $u(t)$ is a solution to the Cauchy problem of \eqref{00NL} with initial datum $u_{0}$.} One natural idea to construct orbitally stable standing waves solutions with prescribed mass for \eqref{00NL} is to consider the following minimization problem
\begin{equation}\label{P1}
\nu_{c}={\inf}\left\{E(u): u\in H^{1}(\Gamma), {u\in S(c)}\right\},
\end{equation}
where
\begin{equation*}
S(c)=\left\{u\in H^{1}(\Gamma): \|u\|^{2}_{L^{2}}=c\right\}.
\end{equation*}

Now if Assumptions 1-3 hold true, then  in the subcritical case $1<p<5$, the energy functional $E(u)$ is bounded from below and $\nu_{c}>-\infty$ for every $c>0$. Furthermore, there exists $c^{\ast}>0$  small enough such that for $0<c<c^{\ast}$ any minimizing sequence for problem \eqref{P1} is precompact in $H^{1}(\Gamma)$. In particular, the set $\mathcal{M}_{c}:=\left\{\varphi \in H^{1}(\Gamma): \|\varphi\|^{2}_{L^{2}}=c, \nu_{c}=E(\varphi)\right\}$ is   $H^{1}(\Gamma)$-stable  with respect to NLS \eqref{00NL}; see \cite[Theorem 1]{CaFiNo2017}  for more details. An analogous result can be proven for the critical case $p=5$  (see \cite[Theorem 2]{Cacci2017}). Other results in this direction, for the NLS
on graphs without confining potentials, were obtained in \cite{ADAM, Adaminew}.

On the other hand, suppose that $\Gamma$ is a star-graph consisting of a central vertex $v_{0}$ and $N$ edges (half-lines) attached to it. If one assumes that $p>5$, then $\nu_{c}=-\infty$. Indeed, first note that there exist constants $a>0$ and $b>0$ such that 
(see Lemma \ref{L1} ii) below)
\begin{equation*}
E(u)\leq a\int_{\Gamma}|u^{\prime}|^{2}dx+b\int_{\Gamma}|u|^{2}dx-\frac{1}{p+1}\int_{\Gamma}|u|^{p+1}dx.
\end{equation*}
Next, if we fix $\phi \in S(c)$ and define $\phi_{\lambda}(x)=\lambda^{\frac{1}{2}}\phi (\lambda x)$, then $\|\phi_{\lambda}\|^{2}_{L^{2}}=\|\phi\|^{2}_{L^{2}}$, 
\begin{equation*}
E(u)\leq a\lambda^{2}\int_{\Gamma}|u^{\prime}|^{2}dx+b\int_{\Gamma}|u|^{2}dx-\frac{\lambda^{\frac{(p-1)}{2}}}{p+1}\int_{\Gamma}|u|^{p+1}dx\rightarrow -\infty,
\end{equation*}
as $\lambda\rightarrow \infty$ and hence $\nu_{c}=-\infty$. For this reason, in the supercritical case $p>5$,  it is not convenient to consider  the  minimization problem \eqref{P1} to construct normalized solutions.

The purpose of this paper is to complement the existence and stability results of Cacciapuoti et al. \cite{CaFiNo2017, Cacci2017} by considering the supercritical case $p>5$. Following the ideas developed in \cite{Luo2018, BEBOJEVI2017}, we introduce a local minimization problem. Indeed, set 
\begin{equation}\label{n1}
B(r):=\left\{u\in H^{1}(\Gamma): \|u\|^{2}_{G}:=\mathfrak{F_{}}[u]+2\lambda_{0}\|u\|^{2}_{L^{2}}\leq r\right\}.
\end{equation}
If assumptions 2-3 hold true, then the energy functional $E$ restricted to $S(c)\cap B({r})$ is bounded from below (see Lemma \ref{L3} below). 
Thus for every $r>0$, we consider the following  local minimization problem on the metric graph $\Gamma$,
\begin{equation}\label{P2}
\nu^{r}_{c}={\inf}\left\{E(u): u\in H^{1}(\Gamma), {u\in S(c)\cap B({r})}\right\},
\end{equation}
and we denote the set of nontrivial minimizers of  \eqref{P2} by
\begin{equation*}
\mathcal{M}^{r}_{c}=\left\{\varphi\in H^{1}(\Gamma):\varphi \in S(c)\cap B({r}),\, \nu^{r}_{c}=E(\varphi)\right\}.
\end{equation*}

Now we are ready to state our first result.
\begin{theorem} \label{TT1}
Let $p\geq 5$ and assumptions 1-3 hold true. For every $r>0$, there exists $c^{\ast}=c^{\ast}(r)>0$ such that:\\
{i)}  $S(c)\cap B(r)\neq \emptyset$ and $\nu^{r}_{c}>-\infty$ for every $c<c^{\ast}$.\\
{ii)} For any $c<c^{\ast}$ there exists $u\in H^{1}(\Gamma)$ with $u\in  S(c)\cap B(r)$  such that  $E(u)=\nu^{r}_{c}$. In particular,  this implies that $\mathcal{M}^{r}_{c}$ is not the empty set.
\end{theorem}
The key ingredient in the proof of the above result is the concentration compactness method for starlike structures \cite{CaFiNo2017}. We remark that the same local minimization problem was exploited, in the case of NLS on bounded domains of $\mathbb{R}^{N}$, in \cite{NTV, PV2}.

Notice that Theorem \ref{TT1} implies that every minimizing sequence $\left\{u_{n}\right\}$ of $\nu^{r}_{c}$ is relativity compact in $H^{1}(\Gamma)$. We also note that for any $r_{1}$, $r_{2}>0$ and sufficiently small $c>0$, we have that $\mathcal{M}^{r_{1}}_{c}=\mathcal{M}^{r_{2}}_{c}$. Indeed, just to fix the ideas, assume that $r_{1}>r_{2}>0$. It is clear that $\nu^{r_{1}}_{c}\leq \nu^{r_{2}}_{c}$. Moreover, if $u\in S(c)\cap B(r_{1})$ and $E(u)=\nu^{r_{1}}_{c}$, then taking $c$ sufficiently small in Lemma \ref{L4} below, we get $u\in B(r_{2})$. Therefore, $\nu^{r_{1}}_{c}\geq \nu^{r_{2}}_{c}$,  which implies that $\nu^{r_{1}}_{c}= \nu^{r_{2}}_{c}$. In particular, since $\mathcal{M}^{r}_{c}\subset B(rc)$
 (see \eqref{E12}) for $c$ sufficiently small, it follows easily that $\mathcal{M}^{r_{1}}_{c}=\mathcal{M}^{r_{2}}_{c}$.

By a standing wave, we mean a solution of \eqref{00NL} with the form $u(x,t)=e^{i\omega t}\varphi(x)$, where $\omega>0$ and 
$\varphi(x)$ should satisfy the following elliptic equation 
\begin{equation}\label{EP}
H\varphi+\omega\varphi-|\varphi|^{p-1}\varphi=0. 
\end{equation}

In the following theorem we give some properties about the structure of  $\mathcal{M}^{r}_{c}$.
\begin{theorem} \label{TT2}
Let $p\geq 5$ and assumptions 1-3 hold true. Then for every fixed $r>0$ and $c<c^{\ast}(r)$ we have:\\
{i)}  For any $\varphi\in\mathcal{M}^{r}_{c}$, there exists $\omega\in  \mathbb{R}$ such that $u(x,t)=e^{i\omega t}\varphi(x)$ is a standing wave solution to NLS \eqref{00NL} with the estimates
\begin{equation*}
\lambda_{0}<\omega\leq\lambda_{0}(1+Kc^{\frac{p-1}{2}}).
\end{equation*}
where $K$ is a positive constant. Notice in particular that $\varphi$ is a solution of the stationary problem \eqref{EP}, and 
$\omega\rightarrow \lambda_{0}$ as $c\rightarrow 0$. \\
{ii)}Suppose that $W\leq 0$.  If $\varphi\in\mathcal{M}^{r}_{c}$, then there exists $\theta\in \mathbb{R}$ such that $\varphi(x)=e^{i\theta}\rho(x)$, where 
$\rho$ is a positive function on $\Gamma$.
\end{theorem}
The following orbital stability result follows from Theorems \ref{TT1} and \ref{TT2}.
\begin{corollary} \label{TT33}
Let $p\geq 5$. Then for every fixed $r>0$ and $c<c^{\ast}(r)$ we have that  $\mathcal{M}^{r}_{c}$ is  $H^{1}(\Gamma)$-stable with respect to NLS \eqref{00NL}.
\end{corollary}

Now assume that $\Gamma$ is a star-graph with $N$ edges. If  $W=0$ and $\alpha_{v}=\gamma >0$ in \eqref{NN1}-\eqref{NN2}, then it is well known that equation \eqref{EP} has $[(N-1)/2]+1$ (here $[s]$ denote the integer part of $s\in \mathbb{R}$) solutions $\phi_{\omega,j}=(\varphi_{\omega,j})^{N}_{j=1}$ with $j=0$, $1$, $\ldots$, $[(N-1)/2]$, which are given by

\begin{equation}\label{DF1}
(\varphi_{\omega,j})_{i}(x)=
\begin{cases} 
f(x-a_{j}) \quad i= 1,\ldots,j \\
f(x+a_{j}) \quad i= j+1,\ldots,N
\end{cases}
a_{j}=\text{tanh}^{-1}\left(\frac{\gamma  }{(N-2j)\sqrt{\omega}}\right)
\end{equation}
where
\begin{align*}
f(x)=\left[\frac{(p+1)\omega}{2}\text{sech}^{2}\left(\frac{(p-1)\sqrt{\omega}}{2}x\right) \right]^{\frac{1}{p-1}}, \quad \omega>\frac{\gamma^{2}}{(N-2j)^{2}}.
\end{align*}
Moreover, in this case (see \textit{Assumption 3}),
\begin{align*}
-\lambda_{0}=-\frac{\gamma^{2}}{N^{2}}=\text{inf}\left\{\mathfrak{F_{\gamma}}[u]: u\in H^{1}(\Gamma), \|u\|^{2}_{L^{2}}=1\right\}.
\end{align*}
In \cite{AQFF} (see also \cite{ANGO}) the authors study the stability of $u(x,t)=e^{i\omega t}\phi_{\omega,0}(x)$  in the $L^{2}$-critical and supercritical cases.  Notice that the stability analysis in \cite{AQFF, ANGO} relies on the theory by Grillakis, Shatah and Strauss.
 
In the  $L^{2}$-supercritical case $p>5$, we apply the Theorem \ref{TT2} and Corollary \ref{TT33} in order to deduce directly the
stability of the  standing wave $u(x,t)=e^{i\omega t}\phi_{\omega,0}(x)$, which was previously treated in the literature
only with the Grillakis-Shatah-Strauss theory.
{\begin{corollary}\label{C11} Let $\Gamma$ be a star graph.  Assume  that $W=0$ and $\alpha=\gamma >0$ in \eqref{NN1}-\eqref{NN2}. If $p\geq 5$,  then there exists $\omega^{\ast}>0$ such that $e^{i\omega t}\phi_{\omega,0}$ is stable in $H^{1}(\Gamma)$ for any $\omega\in (\gamma^{2}/N^{2}, \omega^{\ast})$. 
\end{corollary}}

The plan of the paper is the following. In Section \ref{S:1} we recall several known results, which will be needed later. In Section \ref{S:2}, we prove the existence of a minimizer for $\nu^{r}_{c}$.  Section \ref{S:3} contains the proof of Theorem \ref{TT2}. Section \ref{S:5} is  devoted to the proof of Corollaries \ref{TT33} and \ref{C11}. Throughout this paper, the letter $C$ may stand for various strictly positive
constants, when no confusion is possible.

\section{Preliminary}
\label{S:1}
We recall that a function $u$ on $\Gamma$ is a collection of functions $u_{e}(x)$ defined on each edge $e\in J$. On $\Gamma$ we consider the Hilbert space $L^{2}(\Gamma)$ of measurable and square-integrable functions $u:\Gamma\rightarrow \mathbb{C}^{|J|}$, equipped
with the standard norm,
\begin{equation*}
L^{2}(\Gamma)= \bigoplus_{e\in J}L^{2}(I_{e}),\quad\quad \left\|u\right\|^{2}_{L^{2}}=\sum_{e\in J}\int_{I_{e}}|u_{e}(x)|^{2}dx,
\end{equation*}
where $u_{e}\in L^{2}(I_{e})$. Analogously, we define the $L^{p}(\Gamma)$-spaces for $1\leq p <+\infty$ as,
\begin{equation*}
L^{p}(\Gamma)= \bigoplus_{e\in J}L^{p}(I_{e}),\quad\quad \left\|u\right\|^{p}_{L^{p}}=\sum_{e\in J}\|u_{e}\|^{p}_{L^{p}}.
\end{equation*}

From Proposition 2.2 in \cite{CaFiNo2017} we have that $-\lambda_{0}$ is a simple eigenvalue of $H$ corresponding to strictly positive eigenfunction; more precisely:
\begin{lemma} \label{L2}
Let assumptions 1-3 hold true. Then $-\lambda_{0}$ is a simple eigenvalue of $H$ with corresponding strictly positive eigenfunction $\psi_{0}\in H^{1}(\Gamma)$ such that $\|\psi_{0}\|^{2}_{L^{2}}=1$.
\end{lemma}

For convenience, we recall the Gagliardo-Nirenberg inequality on graphs.
\begin{lemma} \label{L1}
i) Let $\Gamma$ be any non-compact graph and $p\geq1$. Then there exists a positive constant $C$ such that for all $v\in H^{1}(\Gamma)$,
\begin{equation}\label{GN}
\|v\|^{p+1}_{L^{p+1}}\leq C\|v^{\prime}\|^{\frac{p-1}{2}}_{L^{2}}\|v\|^{\frac{p+3}{2}}_{L^{2}}.
\end{equation}
ii) Under assumption 2, there exist positive constants $0<C_{1}<1$ and $C_{2}>0$ such that
\begin{equation}\label{LW}
\left|(u, Wu)-\sum_{v\in V}\alpha_{v}|u(v)|^{2}\right|\leq C_{1}\|u^{\prime}\|^{2}_{L^{2}}+C_{2}\|u\|^{2}_{L^{2}}.
\end{equation}
In particular,  there exists a constant $K>0$ such that $\|u\|^{2}_{H^{1}}\leq K(\|u\|^{2}_{G}+\lambda_{0}\|u\|^{2}_{L^{2}})$.
Here, the norm $\|\cdot\|^{2}_{G}$ is defined in \eqref{n1}.
\end{lemma}
See \cite[Remark 2.1]{CaFiNo2017} for the proof of \eqref{LW}. For a proof of Gagliardo-Nirenberg inequality \eqref{GN} we refer to \cite[Proposition 2.1]{ADAM}.

\section{Variational analysis} 
\label{S:2}
In this section, we give the proof of Theorem \ref{TT1}. We have divided the proof into a sequence of lemmas.
\begin{lemma} \label{L3}
i)Let $r>0$ be fixed, then $S(c)\cap B(r)$ is not empty set iff $c\leq \frac{r}{\lambda_{0}}$.\\
ii) For any $r$, $c>0$, if $S(c)\cap B(r)\neq \emptyset$, then $\nu^{r}_{c}>-\infty$.
\end{lemma}
\begin{proof}
Let $u=\sqrt{c}\psi_{0}$ with $c\leq \frac{r}{\lambda_{0}}$. From Lemma \ref{L2}, we see that $\|u\|^{2}_{L^{2}}=c$ and 
$$\|u\|^{2}_{G}=\mathfrak{F_{}}[\sqrt{c}\psi_{0}]+2\lambda_{0}\|\sqrt{c}\psi_{0}\|^{2}_{L^{2}}=-\lambda_{0}c^{}\|\psi_{0}\|^{2}_{L^{2}}+2\lambda_{0}c^{}\|\psi_{0}\|^{2}_{L^{2}}=\lambda_{0}c^{}\leq r,$$ this implies that $S(c)\cap B(r)\neq \emptyset$. On the other hand, if $u\in S(c)\cap B(r)\neq \emptyset$, we see that  
$$r\geq \|u\|^{2}_{G}=\mathfrak{F_{}}[u]+2\lambda_{0}\|u\|^{2}_{L^{2}}\geq -\lambda_{0}\|u\|^{2}_{L^{2}}+2\lambda_{0}\|u\|^{2}_{L^{2}} = \lambda_{0}c.$$
Therefore, $c\leq \frac{r}{\lambda_{0}}$. This completes the proof of statement i). 
 {
Next, let $u\in S(c)\cap B(r)$. First, notice that from Gagliardo-Nirenberg inequality \eqref{GN} and Lemma \ref{L1} (ii), we see that 
\begin{equation*}
\|u\|^{p+1}_{L^{p+1}}\leq C\|u\|^{\frac{p-1}{2}}_{H^{1}}\|u\|^{\frac{p+3}{2}}_{L^{2}} \leq C_{1}\|u\|^{\frac{p-1}{2}}_{G}\|u\|^{\frac{p+3}{2}}_{L^{2}}+C_{2}\|u\|^{{p+1}}_{L^{2}}.
\end{equation*}
Since $u\in S(c)\cap B(r)$, it follows that
\begin{align*}
E(u)&=\frac{1}{2}\|u\|^{2}_{G}-{\lambda_{0}}\|u\|^{2}_{L^{2}}-\frac{1}{p+1}\|u\|^{p+1}_{L^{p+1}}\\
    &\geq -{\lambda_{0}}\|u\|^{2}_{L^{2}}-C_{1}\|u\|^{\frac{p-1}{2}}_{G}\|u\|^{\frac{p+3}{2}}_{L^{2}}-C_{2}\|u\|^{{p+1}}_{L^{2}}\\ 
  &\geq -{\lambda_{0}}c-Cr^{\frac{p-1}{4}}c^{\frac{p+3}{4}}-C_{2}c^{\frac{p+1}{2}}>-\infty.
\end{align*}
This ends the proof.}
\end{proof}

\begin{lemma} \label{L4}
Let $p\geq 5$.  Then for any number $r>0$, there exists $c^{\ast}:=c^{\ast}(r)>0$ such that for every $c<c^{\ast}$, $S(c)\cap B(r)\neq \emptyset$ and 
\begin{equation}\label{E12}
\inf_{S(c)\cap B(rc/2)} E(u)< \inf_{S(c)\cap (B(r)\setminus B(rc))} E(u).
\end{equation}
\end{lemma}
\begin{proof}
From Lemma \ref{L3} i), we have that if $c\leq r/\lambda_{0}$, then $S(c)\cap B(r)\neq \emptyset$. By  Gagliardo-Nirenberg inequality \eqref{GN} we have
\begin{equation*}
\|u\|^{p+1}_{L^{p+1}}\leq C_{1}\|u\|^{\frac{p+3}{2}}_{L^{2}}\|u\|^{\frac{p-1}{2}}_{G}+C_{2}\|u\|^{p+1}_{L^{2}},
\end{equation*}
for positive constants $C_{1}$ and $C_{2}$. Now, if $u\in S(c)$,  it follows that
\begin{equation}\label{NLSKDV}
\begin{cases} 
 E(u)\geq \frac{1}{2}\|u\|^{2}_{G}-C_{1}c^{\frac{p+3}{4}}\|u\|^{\frac{p-1}{2}}_{G}-C_{2}c^{\frac{p+1}{2}}-\lambda_{0}c\\  
 E(u)\leq \frac{1}{2}\|u\|^{2}_{G}-\lambda_{0}c
\end{cases} 
\end{equation}
{
Set
\begin{equation*}
\begin{cases} 
n_{c}(s)=\frac{1}{2}s^{}-C_{1}c^{\frac{p+3}{4}}s^{1+\epsilon}-C_{2}c^{\frac{p+1}{2}}-\lambda_{0}c, \quad \text{with}\quad\epsilon=\frac{p-5}{4} \\
m_{c}(s)=\frac{1}{2}s^{}-\lambda_{0}c.
\end{cases} 
\end{equation*}
Notice that $\epsilon\geq 0$ because $p\geq 5$.
Next, it is clear that if there exists $c^{\ast}:=c^{\ast}(r)>0$ such that,
\begin{equation}\label{Ep1}
 m_{c}\left(\frac{rc}{2}\right)< \inf _{s\in (rc,r)}n_{c}(s)\quad \text{ for any $c<c^{\ast}(r)<<1$},
\end{equation}
then this implies \eqref{E12}. Note that $n_{c}\in C^{2}([0,\infty))$ and $n^{\prime}_{c}(s)=(\frac{1}{2}-C_{1}(1+\epsilon)c^{\frac{p+3}{4}}s^{\epsilon})>0$ for $s\in (0, r)$ and for $c<c^{\ast}_{1}(r)<<1$. Therefore,
\begin{equation*}
\inf _{s\in (rc,r)}n_{c}(s)=n_{c}(rc)=\frac{1}{2}(rc)-C_{1}c^{\frac{p+3}{4}}(rc)^{1+\epsilon}-C_{2}c^{\frac{p+1}{2}}-\lambda_{0}c.
\end{equation*}
Finally, since $p\geq 5$,
\begin{align}
\inf _{s\in (rc,r)}n_{c}(s)-m_{c}\left(\frac{rc}{2}\right)&=c\left(\frac{1}{4}r-C_{1}c^{\frac{p+3}{4}+\epsilon}r^{1+\epsilon}-C_{2}c^{\frac{p-1}{2}}\right)\nonumber \\
&>0 \quad \text{  for every $c<c^{\ast}_{2}(r)\leq c^{\ast}_{1}(r)$}.\label{Ep}
\end{align}
Combining \eqref{Ep1} with \eqref{Ep}, we obtain \eqref{E12}, if $c<c^{\ast}:=\min\left\{r/\lambda_{0}, c_{2}^{\ast}(r)\right\}$.
Lemma \ref{L3} is thus proved.}
\end{proof}

\begin{lemma} \label{L5}
Let $r>0$ and $c\leq r/\lambda_{0}$. Let $\left\{u_{n}\right\}\subset H^{1}(\Gamma)$ be a minimizing sequence for $\nu^{r}_{c}$. That is, 
\begin{equation*}
u_{n}\in S(c)\cap B(r), \quad \lim_{n\rightarrow\infty}E(u_{n})=\inf_{{u\in S(c)\cap B({r})}}E(u)=\nu^{r}_{c}.
\end{equation*}
Then
\begin{equation}\label{CT}
\liminf_{n\rightarrow\infty}\|u_{n}\|^{p+1}_{L^{p+1}}>0.
\end{equation}
\end{lemma}
\begin{proof}
First, let us show that 
\begin{equation}\label{DE}
\nu^{r}_{c}=\inf_{{u\in S(c)\cap B({r})}}E(u)< -\frac{\lambda_{0}}{2}c.
\end{equation}
Indeed, we set $\psi_{c}=\sqrt{c}\psi_{0}$, where $\psi_{0}$  is defined in Lemma \ref{L2} and $c\leq r/\lambda_{0}$. It is clear that $\|\psi_{c}\|^{2}_{L^{2}}=c\|\psi_{0}\|^{2}_{L^{2}}=c$ and 
\begin{align*}
\|\psi_{c}\|^{2}_{G}&=c\mathfrak{F}[\psi_{0}]+2\lambda_{0}c\|\psi_{0}\|^{2}_{L^{2}}\\
&=-c\lambda_{0}+2\lambda_{0}c=c\lambda_{0}\leq r,
\end{align*}
thus $\psi_{c}\in S(c)\cap B(r)$ and 
\begin{align*}
\inf_{{u\in S(c)\cap B({r})}}E(u)&\leq E(\psi_{c})=\frac{c}{2}\mathfrak{F}[\psi_{0}]-\frac{1}{p+1}\|\psi_{c}\|^{p+1}_{L^{p+1}}\\
&=-\frac{\lambda_{0}}{2}c-\frac{1}{p+1}\|\psi_{c}\|^{p+1}_{L^{p+1}}<-\frac{\lambda_{0}}{2}c.
\end{align*}
On the other hand, assume by the absurd that $\liminf_{n\rightarrow\infty}\|u_{n}\|^{p+1}_{L^{p+1}}=0$. Then
\begin{equation*}
\nu^{r}_{c}=\lim_{n\rightarrow\infty}\frac{1}{2}\mathfrak{F}[u_{n}]\geq -\frac{\lambda_{0}}{2}c,
\end{equation*}
which is a contradiction since $\nu^{r}_{c}<-\frac{\lambda_{0}}{2}c$. This completes the proof of lemma.
\end{proof}

For any $y\in \Gamma$ and $t>0$, we denote by $B(y,t)$ the open ball of center $y$ and  radius $t$,
\begin{equation*}
B(y,t):=\left\{x\in \Gamma : d(x,y)<t\right\}.
\end{equation*}
Here,  $d(x, y)$ denotes the distance between two points of the graph (see \cite[Section 3]{CaFiNo2017}).

To each minimizing sequence $\left\{u_{n}\right\}_{n\in \mathbb{N}}\subset H^{1}(\Gamma)$ of $\nu^{r}_{c}$, we define the following
sequence of functions (L\'evy concentration functions) $M_{n}$,
\begin{equation*}
M_{n}(t)=\sup_{y\in B(y,t)} \|u_{n}\|_{L^{2}(B(y,t))}^{2}.
\end{equation*}
Let
\begin{equation*}
\tau=\lim_{t\rightarrow\infty}\lim_{n\rightarrow\infty} M_{n}(t).
\end{equation*}
Since $\|u_{n}\|^{2}_{L^{2}}=c$, it is clear that $0\leq \tau\leq c$. Concentration compactness lemma for starlike structures \cite{CaFiNo2017} shows
that there are three (mutually exclusive) possibilities for $\tau$,\\
{\it (i)} (Vanishing) $\tau=0$. Then, up to a subsequence, $\|u_{n}\|^{p+1}_{L^{p+1}}\rightarrow  0$ as $n\rightarrow \infty$ for all 
$2<p\leq \infty$.\\
{\it (ii)} (Dichotomy) $\tau\in (0,c)$. Then, there exist $\left\{v_{{k}}\right\}_{k\in \mathbb{N}}$, $\left\{w_{{k}}\right\}_{k\in \mathbb{N}}\subset H^{1}(\Gamma)$ such that
\begin{align}
 &\text{supp}\,v_{{k}} \cap \text{supp}\,w_{{k}}=\emptyset  \\
 &|v_{{k}}(x)| + |w_{{k}}(x)| \leq  |u_{{k}}(x)| \text{ for all $x\in \Gamma$} \\
 &\|v_{k}\|_{H^{1}}+ \|w_{k}\|_{H^{1}} \leq C  \|u_{n_{k}}\|_{H^{1}}          \\
 &\|v_{k}\|^{2}_{L^{2}}\rightarrow \tau \quad  \|w_{k}\|^{2}_{L^{2}}\rightarrow c-\tau          \\
 &\liminf_{k\rightarrow \infty}(\|u^{\prime}_{n_{k}}\|^{2}_{L^{2}}-\|v^{\prime}_{{k}}\|^{2}_{L^{2}}- \|w^{\prime}_{{k}}\|^{2}_{L^{2}})\geq 0 \\
 &\lim_{k\rightarrow \infty}(\|u_{n_{k}}\|^{p+1}_{L^{p+1}}-\|v_{{k}}\|^{p+1}_{L^{p+1}}- \|w_{{k}}\|^{p+1}_{L^{p+1}})= 0 \quad 2\leq p <\infty\\
 &\lim_{k\rightarrow \infty}\| |u_{n_{k}}|^2-|v_{{k}}|^2- |w_{{k}}|^2\|_{L^{\infty}}= 0.                         
\end{align}
{\it (iii)} (Compactness) $\tau=c$. Then, up to a subsequence,  at least one of the two following cases occurs,\\
(Convergence) There exists $u\in H^{1}(\Gamma)$ such that $u_{n}\rightarrow u$ in $L^{p}$ for all $2\leq p \leq \infty$.\\
(Runaway) There exists $e^{\ast}\in J^{ex}$, such that for any $t>0$ and $2\leq p \leq \infty$
\begin{equation}\label{RW}
\lim_{n\rightarrow\infty}\sum_{e\neq e^{\ast}}\left(\|(u_{n})_{e}\|^{p}_{L^{p}(I_{e})}+\|(u_{n})_{e^{\ast}}\|^{p}_{L^{p}((0,t))}\right)=0.
\end{equation}

Now we give the proof of Theorem \ref{TT1}.
\begin{proof}[\bf {Proof of Theorem \ref{TT1}}] The statement i) follows  from Lemma \ref{L3}.  
Next we prove statement ii) of theorem. Let $\left\{u_{n}\right\}$ be a minimizing sequence of  $\nu^{r}_{c}$, then $\left\{u_{n}\right\}$ is bounded in $H^{1}(\Gamma)$. Indeed, since $u_{n}\in S(c)\cap B(r)$, from Lemma \ref{L1} ii) we see easily that the sequence $\left\{u_{n}\right\}$ is bounded in $H^{1}(\Gamma)$.
Moreover, since $H^{1}(\Gamma)$ is a Hilbert space, there is $u\in H^{1}(\Gamma)$ such that, up to a subsequence, $u_{n}\rightharpoonup u$ in $H^{1}(\Gamma)$ and  $(u_{n})_{e}(x)$ converges to $(u)_{e}(x)$ a.e  $x\in I_{e}$, $e\in J$. Next we analyze separately the three possibilities: $\tau=0$, $0<\tau <c$ and $\tau=c$. 
From Lemma \ref{L5}, it follows that $\tau>0$. Therefore, the possibility of ``vanishing'' is ruled out. Now following the same argument as in the proof of Theorem 1 in \cite{CaFiNo2017} we can rule out the possibility of dichotomy; that is $\tau\notin (0,c)$. Thus, we see that $\tau=c$.

Next we prove that for $c<c^{\ast}$  the minimizing sequence  $\left\{u_{n}\right\}$ is not runaway. We argue by contradiction. Suppose that
$\left\{u_{n}\right\}$ is runaway. From \eqref{RW} we see that $\lim_{n\rightarrow \infty} \|(u_{n})_{e}\|^{p}_{L^{p}}=0$ for all $e\neq e^{\ast}$, Moreover, since $W_{-}\in L^{r}(\Gamma)$ with $1\leq r\leq1+2/(p-1)$, following the same ideas as in \cite[Theorem 1]{CaFiNo2017}, one can get that
\begin{equation*}
\lim_{n\rightarrow \infty}|u_{n}(v)|=0 \quad \text{for all $v\in V$ and} \quad  \lim_{n\rightarrow \infty}(u_{n}, W_{-}u_{n})=0.
\end{equation*}
Thus, we have
\begin{equation}\label{Ec1}
\lim_{n\rightarrow \infty}E(u_{n})\geq \lim_{n\rightarrow \infty}\left[\int^{\infty}_{0}|(u_{n})_{e^{\ast}}^{\prime}|^{2}dx-\frac{1}{p+1}\int^{\infty}_{0}|(u_{n})_{e^{\ast}}|^{p+1}dx\right].
\end{equation}
On the other hand, since $u_{n}\in S(c)\cap B(r)$, using the inequality of Gagliardo-Nirenberg it is not difficult to show that there exists a constant $K$ independent of $n$ such that 
\begin{equation}\label{Gc1}
\|u_{n}\|^{p+1}_{L^{p+1}}\leq Kr^{(p-1)/4}c^{(p+3)/4}.
\end{equation}
Therefore, combining \eqref{DE}, \eqref{Ec1} and \eqref{Gc1}, we obtain
\begin{equation*}
\inf_{{u\in S(c)\cap B({r})}}E(u)< -\frac{\lambda_{0}}{2}c<-Kr^{(p-1)/4}c^{(p+3)/4}\leq \lim_{n\rightarrow \infty}E(u_{n}),
\end{equation*}
for sufficiently small $c^{\ast}(r)<<1$ and $c<c^{\ast}(r)$, which is a contradiction and hence, every minimizing sequence for $\nu^{r}_{c}$ must be compact if $c<c^{\ast}(r)$. Thus, we have that  $u_{n}$
converges, up to taking subsequences, in $L^{p}$-norm to the function $u$ satisfying $\|u\|^{2}_{L^{2}}=c$. Consequently, from the weak convergence in $H^{1}(\Gamma)$, it follows that
\begin{equation*}
E(u)\leq \lim_{n\rightarrow \infty}E(u_{n})=\nu^{r}_{c}, \quad \text{and} \quad \|u\|^{2}_{G}\leq \liminf_{n\rightarrow \infty}\|u_{n}\|^{2}_{G}.
\end{equation*}
Thus, $u_{}\in S(c)\cap B(r)$ and $E(u)=\nu^{r}_{c}$. Now, if we assume that $\|u\|^{2}_{H^{1}}< \liminf_{n\rightarrow \infty}\|u_{n}\|^{2}_{H^{1}}$, then $E(u)<\nu^{r}_{c}$,  which is absurd, and thus we deduce that $u_{n}\rightarrow u$ strongly in $H^{1}(\Gamma)$.  This finishes the proof.



\end{proof}

\section{Proof of Theorem \ref{TT2}}
\label{S:3}

The aim of this section is to prove Theorem \ref{TT2}. First, we need the following lemma.

\begin{lemma} \label{L8}
Let $u\in H^{1}(\Gamma)$ be a solution of \eqref{EP}. Then, for every $e\in J$, the restriction $u_{e}:[0, L_{e})\rightarrow \mathbb{C}$ of $u$ to the $e$-th edge satisfies the following properties:
\begin{align}
& u_{e}\in H^{2}((0,L_{e}))\cap C^{2}((0,L_{e})),  \label{1} \\
& -u^{\prime\prime}_{e}+W_{e}u_{e}+\omega u_{e}-|u_{e}|^{p-1}u=0\, \quad \mbox{on}\quad (0,L_{e}),\label{2}\\
& \sum_{e\prec v}\partial_{o}u_{e}(v)=-\alpha_{v}u_{e}(v)\quad \text{ for all $v\in V$}\label{3},
\end{align}
where $W_{e}$ is the component of the potential $W$ on the edge $e$.
\end{lemma}
\begin{proof}
Fix $l\in J$. Statements \eqref{1} and \eqref{2} are derived from a standard bootstrap argument using test functions $\zeta\in C^{\infty}_{0}((0, L_{e}))$. 
Indeed, by \eqref{EP} applied with $\varphi=(\varphi_{e})_{e\in J}$, where $\varphi_{l}=\zeta u_{l}$ and $\varphi_{e}=0$ for $e\neq l$, we get
\begin{equation*}
-(\zeta u_{l})^{\prime\prime}+\omega \zeta u_{l}=-W_{e}\zeta u_{l}-\zeta u^{\prime\prime}_{l}-2\zeta^{\prime} u^{\prime}_{l}+\zeta |u_{l}|^{p-1}u_{l}
\end{equation*}
in the sense of distributions on $(0, L_{e})$. Now, since  the right side is in $L^{2}((0,L_{{e}}))$, it follows that $u_{l}\in H^{2}((0,L_{{e}}))$, and hence $u_{l}\in H^{2}((0,L_{{e}}))\cap C^{1}((0,L_{{e}}))$. A similar argument shows that $u_{l}\in  C^{2}((0,L_{{e}}))$.
In particular, statement \eqref{2} is true.
Finally, a standard argument shows that the solution $u$ of the stationary problem \eqref{EP}  is an element of the domain of the operator $H$, 
and it satisfies the boundary conditions \eqref{3} at the vertex $v$.  For more details see, for example \cite[Theorem 4]{AQFF} and \cite[Lemma 4.1]{AA2017}.
\end{proof}

\begin{proof}[\bf {Proof of Theorem \ref{TT2}}]
Our proof is inspired by \cite[Theorem 1 ]{BEBOJEVI2017} and \cite[Proposition 2]{CaFiNo2017}. First we remark that if $u\in\mathcal{M}^{r}_{c}$, then from Lemma \ref{L4} we have that $u\in B(rc)$, that is, $u$ stays away from the boundary of $S(c)\cap B(r)$. Therefore, we see that $u$ is a critical point of $E$ on $S(c)$ and  there exists a Lagrange multiplier $\omega\in \mathbb{R}$ such that 
\begin{equation}\label{Pl}
Hu+\omega u-|u|^{p-1}u=0.
\end{equation}
 Multiplying \eqref{Pl} with $\overline{u}$, and integrating over $\Gamma$ we obtain
\begin{align*}
-\omega \|u\|^{2}_{L^{2}}&=\mathfrak{F}[u]-\|u\|^{p+1}_{L^{p+1}}=2E(u)+\frac{2}{p+1}\|u\|^{p+1}_{L^{p+1}}-\|u\|^{p+1}_{L^{p+1}}\\
&=2E(u)+\frac{1-p}{p+1}\|u\|^{p+1}_{L^{p+1}}<2E(u).
\end{align*}
Thus, from \eqref{DE}, we see that 
\begin{equation*}
\omega>-\frac{2 E(u)}{c}>\lambda_{0}.
\end{equation*}
Next, since $u\in S(c)\cap B(r)$, by Lemma \ref{L1} ii) we see that there exists a constant $K>0$ such that
\begin{equation*}
\|u\|^{p+1}_{L^{p+1}}\leq K(c^{\frac{p+3}{4}}\|u\|^{\frac{p-1}{2}}_{G}+\lambda_{0}c^{\frac{p+1}{2}}).
\end{equation*}
Moreover, notice that
\begin{equation*}
\|u\|^{2}_{G}=\mathfrak{F}[u]+2\lambda_{0}\|u\|^{2}_{L^{2}}\geq   \lambda_{0}\|u\|^{2}_{L^{2}}=\lambda_{0}c,
\end{equation*}
which implies that
\begin{align*}
-\omega c&=\mathfrak{F}[u]-\|u\|^{p+1}_{L^{p+1}}=\|u\|^{2}_{G}-\|u\|^{p+1}_{L^{p+1}}-2\lambda_{0}c\\
&\geq \|u\|^{2}_{G}-Kc^{\frac{p+3}{4}}\|u\|^{\frac{p-1}{2}}_{G}-K\lambda_{0}c^{\frac{p+1}{2}}-2\lambda_{0}c\\
&=\|u\|^{2}_{G}(1-Kc^{\frac{p+3}{4}}\|u\|^{\frac{p-5}{2}}_{G})-K\lambda_{0}c^{\frac{p+1}{2}}-2\lambda_{0}c\\
&\geq \lambda_{0}c (1-Kc^{\frac{p+3}{4}}(rc)^{\frac{p-5}{4}})-K\lambda_{0}c^{\frac{p+1}{2}}-2\lambda_{0}c\\
&=\lambda_{0}c (-1-Kc^{\frac{p-1}{2}}r^{\frac{p-5}{4}}-Kc^{\frac{p-1}{2}}).
\end{align*}
It follows that
\begin{equation*}
\omega \leq \lambda_{0} (1+Kc^{\frac{p-1}{2}}+Kc^{\frac{p-1}{2}}r^{\frac{p-5}{4}})=\lambda_{0} (1+Kc^{\frac{p-1}{2}}(1+r^{\frac{p-5}{4}})).
\end{equation*}
Thus, since  $p\geq5$, we obtain the proof of i) of theorem. 

Let $u\in \mathcal{M}^{r}_{c}$ be a complex valued minimizer. Since $\||u|^{\prime}\|^{2}_{L^{2}}\leq \|u^{\prime}\|^{2}_{L^{2}}$, it follows that $|u|\in S(c)\cap B(r)$ and $E(|u|)\leq E(u)=\nu^{r}_{c}$. In particular, $|u|\in \mathcal{M}^{r}_{c}$ and $E(u)=E(|u|)$. This implies that 
\begin{equation}\label{CIT}
\sum_{e\in J}\int^{L_{{e}}}_{0}||u_{e}|^{\prime}(x)|^{2}dx=\sum_{e\in J}\int^{L_{{e}}}_{0}|u^{\prime}_{e}(x)|^{2}dx.
\end{equation}
Now we set $\psi:=|u|$. We claim that $\psi(x)>0$ for all $x\in\Gamma$. 

First, since $\psi\in \mathcal{M}^{r}_{c}$, we obtain that $\psi$ is a critical point of $E$ on $S(c)$. Therefore, there exists a Lagrange multiplier $\omega>\lambda_{0}$ such that 
\begin{equation*}
H\psi+\omega\psi-\psi^{p}=0. 
\end{equation*}
Notice that, by Lemma \ref{L8}, $\psi\in \text{dom}(H)$, $\psi_{e}\in H^{2}((0, L_{e}))\cap C^{2}((0, L_{e}))$ (here $L_{e}=+\infty$ if $e\in J^{ext}$) and  
\begin{equation*}
-\psi^{\prime\prime}_{e}+W_{e}\psi_{e}+\omega\psi_{e}-\psi_{e}^{p}=0, \quad \text{for all $x\in (0, L_{e})$ and $e\in J$.}
\end{equation*}
We recall that $W\leq 0$. We set $B(s):=\omega s-s^{p}$. Since $B\in C^{1}[0,+\infty)$, is nondecreasing for $s$ small, $B(0) = 0$ and $B(\omega^{1/(p-1)})= 0$, by \cite [Theorem 1]{LVL}, it follows that for every $e\in J$,  $\psi_{e}$ is either trivial or strictly positive on $(0, L_{e})$.

Secondly, if we suppose that $\psi_{e}(0)=\psi^{\prime}_{e}(0)=0$, then $\psi_{e}$ is trivial on $[0, L_{e}]$. Indeed, for some $\epsilon>0$, we define 
\begin{equation*}
\tilde{\psi}_{e}(x)=
\begin{cases} 
\psi_{e}(x) \quad \text{if $x\in [0, L_{e})$,}\\
0\quad \text{if $x\in (-\epsilon, 0)$.} 
\end{cases} 
\end{equation*}
Then, by Sobolev extension theorem, $\tilde{\psi}_{e}\in H^{2}((-\epsilon, L_{e}))$ and 
\begin{equation*}
-\tilde{\psi}^{\prime\prime}_{e}+\tilde{W}_{e}\tilde{\psi}_{e}+\omega\tilde{\psi}_{e}-\tilde{\psi}_{e}^{p}=0, \quad \text{for all $x\in (-\epsilon, L_{e})$,}
\end{equation*}
where the function $\tilde{W}_{e}$  is the extension by zero of ${W}_{e}$. Thus, by \cite [Theorem 1]{LVL} and applying the same argument as above, we have that $\tilde{\psi}_{e}=0$ on $(-\epsilon, L_{e})$. Analogously we may consider the case when $\psi_{e}(L_{e})=\psi^{\prime}_{e}(L_{e})=0$.

Finally, we prove $\psi>0$. To begin assume that $\psi(v)=0$ for some $v\in V$. Since $\psi\in \text{dom}(H)$, it follows that $\sum_{v\prec e}\partial_{o}\psi_{e}(v)=0$. Without loss of generality, we can assume that the vertex $v$ coincides with $x=0$. Notice that $\psi^{\prime}_{e}(0)\geq 0$. Indeed, this follows from the fact that $\psi_{e}\in C^{1}([0,L_{e}))$, $\psi_{e}\geq 0$ and $\psi_{e}(0)=0$. Thus we obtain $\psi_{e}(0)=\psi^{\prime}_{e}(0)=0$ due to the boundary conditions at the vertex. Therefore $\psi_{e}=0$ on $(0, L_{e})$ for all $e\prec 0$ and by continuity $\psi=0$ on $\Gamma$, which is a contradiction because $\psi\in \mathcal{M}^{r}_{c}$. This contradiction shows that the supposition is false, and so $\psi_{e}(v)>0$ for all $v\in V$. Therefore, $\psi_{e}(x)>0$ on $[0, L_{e})$ for all $e\prec v$  and, hence, by continuity 
$\psi>0$ on $\Gamma$. This proves the claim.

On the other hand, we can write $u_{e}(x)=\psi_{e}(x)z_{e}(x)$ where $\psi_{e}$, $z_{e}\in C^{1}(0,L_{e})$ and $|u_{e}|=\psi_{e}>0$. Since $|z_{e}|=1$, it follows that 
\begin{equation*}
u^{\prime}_{e}=\psi^{\prime}_{e}z_{e}+\psi_{e}z^{\prime}_{e}=z_{e}(\psi^{\prime}_{e}+\psi_{e}\overline{z_{e}}z^{\prime}_{e}).
\end{equation*}
Notice that $\text{Re}(\overline{z_{e}}z^{\prime}_{e})=0$. This implies that  $|u^{\prime}_{e}|^{2}=|\psi^{\prime}_{e}|^{2}+|\psi_{e}z^{\prime}_{e}|^{2}$. Thus from \eqref{CIT} we see that 
\begin{equation*}
\sum_{e\in J}\int^{L_{{e}}}_{0}|\psi_{e}^{\prime}|^{2}dx=\sum_{e\in J}\int^{L_{{e}}}_{0}|u^{\prime}_{e}|^{2}dx=\sum_{e\in J}\int^{L_{{e}}}_{0}|\psi^{\prime}_{e}|^{2}+\sum_{e\in J}\int^{L_{{e}}}_{0}|\psi_{e}\, z^{\prime}_{e}|^{2}.
\end{equation*}
Using that $\psi_{e}>0$, we obtain that $z^{\prime}_{e}=0$ for every  $e\in J$. Since $z_{e}\in C^{1}(0,L_{e})$, we get $z_{e}(x)=e^{i\theta_{e}}$ on $(0, L_{e})$ with $\theta_{e}=\text{constant}$. Finally, by continuity at the vertex we have that $\theta_{e}=\theta=\text{constant}$ for all  $e\in J$, and this completes the proof of theorem.

\end{proof}

\section{Stability of standing waves}
\label{S:5}

This section is devoted to the proof of  Corollaries \ref{TT33} and \ref{C11}.

\begin{proof}[\bf {Proof of Corollary \ref{TT33}}]
We prove that the set $\mathcal{M}^{r}_{c}$ is  $H^{1}(\Gamma)$-stable with respect to NLS \eqref{00NL}. We argue by contradiction. Suppose that the set $\mathcal{M}^{r}_{c}$ is not stable under flow associated with \eqref{00NL}. Then there exist $\epsilon>0$, a sequence  $\left\{u_{n, 0}\right\}_{n\in \mathbb{N}}\subset H^{1}(\Gamma)$ such that
\begin{equation}\label{Nm}
\inf_{\varphi\in\mathcal{M}^{r}_{c}}\|u_{n, 0}-\varphi\|_{H^{1}}<\frac{1}{n}
\end{equation}
and $\left\{t_{n}\right\}\subset \mathbb{R}^{+}$ such that
\begin{equation}\label{Ctw}
\inf_{\varphi\in \mathcal{M}^{r}_{c}}\|u_{n}(t_{n})-\varphi\|_{H^{1}}\geq \epsilon \quad \text{for all $n\in \mathbb{N}$},
\end{equation}
where $u_{n}$ is the unique solution of the Cauchy problem \eqref{00NL} with initial data $u_{n, 0}$. Without loss of generality, we may assume that $\|u_{n, 0}\|^{2}_{L^{2}}=c$. From \eqref{Nm} and conservation laws, we obtain
\begin{align*}
E(u_{n}(t_{n}))&=E(u_{n, 0})\rightarrow \nu^{r}_{c} \quad \text{as $n\rightarrow\infty$,}\\
\|u_{n}(t_{n})\|^{2}_{L^{2}}&=\|u_{n, 0}\|^{2}_{L^{2}}=c\quad \text{ for all $n$.}
\end{align*}
We note that there exists a subsequence  $\left\{u_{n_{k}}(t_{n_{k}})\right\}$ of  $\left\{u_{n}(t_{n})\right\}$ such that $\|u_{n_{k}}(t_{n_{k}})\|^{2}_{G}\leq r$. Indeed, if $\|u_{n}(t_{n})\|^{2}_{G}> r$  for every $n$ sufficiently large, then by continuity there exists $t^{\ast}_{n}\in (0,t_{n})$ such that $\|u_{n}(t^{\ast}_{n})\|^{2}_{G}=r$. Moreover $\|u_{n}(t^{\ast}_{n})\|^{2}_{L^{2}}=c$ and $E(u_{n}(t^{\ast}_{n}))\rightarrow \nu^{r}_{c}$ as $n\rightarrow \infty$. Therefore, $\left\{u_{n}(t^{\ast}_{n})\right\}_{n\in \mathbb{N}}$ is a minimizing sequence  of $\nu^{r}_{c}$. By Theorem \ref{TT1}, we see that there exists $u^{\ast}\in H^{1}(\Gamma)$ such that $u_{n}(t^{\ast}_{n})\rightarrow u^{\ast}$ strong in $H^{1}(\Gamma)$. In particular,  $E(u^{\ast})=\nu^{r}_{c}$, $\|u^{\ast}\|^{2}_{G}=r$ and $\|u^{\ast}\|^{2}_{L^{2}}=c$, which is a contradiction, because from Lemma \ref{L4} the critical points of $E$  on $S(c)$ do not belong to the boundary of $S(c)\cap B(r)$. 

In conclusion, $\|u_{n_{k}}(t_{n_{k}})\|^{2}_{G}\leq r$, $\|u_{n_{k}}(t_{n_{k}})\|^{2}_{L^{2}}=c$ and $E(u_{n_{k}}(t_{n_{k}}))\rightarrow \nu^{r}_{c}$; that is, $\left\{u_{n_{k}}(t_{n_{k}})\right\}_{k\in \mathbb{N}}$ is a minimizing sequence for $\nu^{r}_{c}$.  Thus, by Theorem \ref{TT1}, up to a subsequence, there exists  a function $\varphi\in\mathcal{M}^{r}_{c}$ such that
\begin{equation*}
\|u_{n_{k}}(t_{n_{k}})-\varphi\|_{H^{1}}\rightarrow 0, \quad \text{as $k\rightarrow\infty$,}
\end{equation*}
which is a contradiction with \eqref{Ctw}.
\end{proof}

\begin{proof}[\bf {Proof of Corollary \ref{C11}}]
Let $r>0$ and $c^{\ast}>0$ be as in Theorem \ref{TT1}.  Notice that $\mathcal{M}^{r}_{c}\neq \emptyset$ for $c< c^{\ast}$.
Let $\psi\in \mathcal{M}^{r}_{c}$. From Theorem \ref{TT2} i)-ii), there exists $\omega>\gamma^{2}/N^{2}$ such that 
\begin{equation*}
H\psi+\omega\psi-|\psi|^{p-1}\psi=0, 
\end{equation*}
and $\omega\rightarrow \gamma^{2}/N^{2}$ as $c\rightarrow 0$.

From Theorem 4 in \cite{AQFF} we have that $$\emptyset\neq\mathcal{M}^{r}_{c}\subset \left\{e^{i\theta}\phi_{\omega,j}: \theta\in \mathbb{R}, j=0, 1, \ldots, [(N-1)/2] \right\},$$ where $\phi_{\omega,j}$ is defined in \eqref{DF1}. We recall that $\phi_{\omega,j}$ is defined for $\omega>{\gamma^{2}}/(N-2j)^{2}$.

Notice that if $c< c^{\ast}$ is sufficiently small, then $\phi_{\omega,j}\notin \mathcal{M}^{r}_{c}$ for $j\geq1$. Indeed, since $\omega\rightarrow \gamma^{2}/N^{2}$ as $c\rightarrow 0$, by taking $c$ sufficiently small we have that $\omega< {\gamma^{2}}/(N-2j)^{2}$ for every $j\geq 1$. That is, there exists $c^{\ast}_{1}$ such that if $c<c^{\ast}_{1}$, then  $\phi_{\omega,j}\notin \mathcal{M}^{r}_{c}$ for $j=1$, $2$, $\ldots$, $[(N-1)/2]$.

On the other hand, we define the function $R(\omega):=\|\phi_{\omega,0}\|^{2}_{L^{2}}$ for $\omega> \gamma^{2}/N^{2}$,  i.e.,
\begin{align*}
R(\omega)&=\frac{2N}{(p-1)}\left(\frac{p+1}{2}\right)^{\frac{2}{p-1}}\omega^{\frac{5-p}{2(p-1)}}\,h\left({\frac{\gamma}{N\omega^{1/2}}}\right),\\
& \text{where} \quad h(x)=\int^{1}_{x}(1-t^{2})^{\frac{3-p}{p-1}}dt.
\end{align*}
Notice that $R(\gamma^{2}/N^{2})=0$.  Moreover, there exists $\omega_{1}^{\ast}>0$ such that  $R^{\prime}(\omega)>0$ for all $\omega\in (\gamma^{2}/N^{2}, \omega_{1}^{\ast})$ (see \cite[Remark 6.1]{AQFF}). This implies that  there exists $c^{\ast}_{2}$ such that $\|\phi_{\omega(c),0}\|^{2}_{L^{2}}=c$ for $0<c\leq c^{\ast}_{2}$ and $\omega(c)\in (\gamma^{2}/N^{2}, \omega_{1}^{\ast})$. 
Without loss of generality we can assume that $r\geq \|\phi_{\omega(c),0}\|^{2}_{L^{2}}$ for $\omega(c)\in (\gamma^{2}/N^{2}, \omega_{1}^{\ast})$. Set $c_{3}^{\ast}:=\min\left\{{c_{1}^{\ast}, c^{\ast}_{2}}\right\}$. It is clear that if $c<c_{3}^{\ast}$, then $\mathcal{M}^{r}_{c}= \left\{e^{i\theta}\phi_{\omega(c),0}: \theta\in \mathbb{R}\right\}$. Therefore, the statement of Corollary \ref{C11} follows from Corollary \ref{TT33}, and this completes the proof.
\end{proof}

\section*{Acknowledgements}
The author wishes to express their sincere thanks to the referees for his/her valuable comments that help to improve the main results of this paper.

\end{document}